\documentclass[aoap,preprint]{imsart}

\RequirePackage{amsthm,amsmath,amsfonts,amssymb}
\RequirePackage[numbers]{natbib}
\RequirePackage{graphicx}

\startlocaldefs

\newtheorem{theorem}{Theorem}[section]
\newtheorem{lemma}[theorem]{Lemma}
\newtheorem{proposition}[theorem]{Proposition}
\newtheorem{corollary}[theorem]{Corollary}
\theoremstyle{remark}

\newcommand\R{{\mathbb R}}
\newcommand\E{{\textnormal E}}
\newcommand\PP{{\textnormal P}}

\endlocaldefs

\begin{document}

\begin{frontmatter}
\title{Asymptotic formula for the conjunction probability of smooth stationary Gaussian fields}
\runtitle{Conjunction probability of smooth stationary Gaussian fields}

\begin{aug}
\author[A]{\fnms{ Viet-Hung} \snm{Pham}\ead[label=e2]{pvhung@math.ac.vn}}

\address[A]{Institute of Mathematics, Vietnam Academy of Science and Technology (VAST),\\
              Hanoi, Vietnam,
\printead{e2}}
\end{aug}

\begin{abstract}
 Let $\{X_i(t):\, t\in S\subset \R^d  \}_{i=1,2,\ldots,n}$ be independent copies of a stationary centered Gaussian field with almost surely smooth sample paths. In this paper, we are interested in the conjunction probability defined as $\PP \left( \exists t\in S: \, X_i(t) \geq u, \, \forall i=1,2,\ldots,n \right)$ for a given threshold level $u$. As $u\rightarrow \infty$, we will provide an asymptotic formula for the conjunction probability. This asymptotic formula is derived from the behaviour of the volume of the set of local maximum points. The proof relies on a result of Aza\"\i s and Wschebor describing the shape of the excursion set of a stationary centered Gaussian field. Our result confirms partially the validity of Euler characteristic method.
\end{abstract}



\end{frontmatter}

\section{Introduction}
Let  $X$  be  a  real  stationary centered Gaussian field with unit variance and almost surely smooth sample paths. Assume more  that it  is defined on a compact set $S\subset \R^d$. Consider  $n$ independent copies $\{X_i(t); i=1,2,\ldots,n\} $ of $X$. In this paper, we are interested in the  behaviour of \textit{conjunction probability}, 
\begin{equation}
\PP \left( \exists t\in S: \, X_i(t) \geq u, \, \forall i \in \overline{1,n}\right),
\end{equation}
where $u$ is a level that will tend to infinity. The probability above  can also be expressed  in different manners:
Equivalently it can be rewritten as  the probability that the conjunction set (excursion set)
$$C_u=\{t\in S:\; X_i(t) \geq u, \, \forall i\in\overline{1,n}\}$$
is non-empty or that   the maximum of the smallest value among the fields exceeds $u$ 
\begin{equation}\label{conju}
\PP \left(\underset{t\in S}{\sup}  \ \underset{1\leq i \leq n}{\min} X_i(t) \geq u \right).
\end{equation}

When $n=1$, the expression (\ref{conju}) is simply the tail distribution of the maximum of a stationary Gaussian field.  Even in this simple case, finding the exact value of the tail distribution is very challenging \cite{MR2478201}. That explains why we  limit our attention to the asymptotic case $u\to  +\infty$.  This problem has been studied extensively in literature. One could mention three main techniques to deal with it: Double-sum method (see \citep{MR0250368,MR1361884}), Euler characteristic method (see \citep{MR2319516,MR2206344,MR2150192}) and Rice method (see \citep{MR1965978,MR2428714,MR2478201}). See also \citep{MR1207215,MR1910648}.

The first method was introduced by Pickands \cite{MR0250368} for stationary Gaussian  "$\alpha$ processes" and later  was  extended to non-stationary  processes  and to  non-Gaussian one by Piterbarg \cite{MR1361884}. Note that  the result depends on some non-explicit constants:  
 \textit{the Pickands constant} depending on the local  self-similarity  exponent $\alpha$  of  the  process. 
 
 The second method was provided by Adler and Taylor \cite{MR2319516} and concerns differentiable processes. It is an important tool  for studying  the geometry of random surfaces. The main idea of this method is as follows:  when the level $u$ is large, the excursion set $C_u$, if  non-empty, is, in most of the cases, a simply-connected domain. Therefore its Euler characteristic, denoted by $\mu_0(C_u)$, is  often equal to 1. Since  most of the values of $\mu_0(C_u)$ are  0 or 1, its expectation could be used as an approximation to the excursion probability.  Adler and Taylor gave that
\begin{equation}\label{EC1}
\E(\mu_0(C_u))=\sum_{i=0}^d \rho_i \mu_i(S),
\end{equation}
where $\rho_i$'s are the Euler characteristic densities defined as
$$\rho_0= \overline{\Phi}(u)=\int_u^{\infty} \frac{e^{-x^2/2}}{\sqrt{2\pi}}dx,$$
$$\rho_i=(2\pi)^{-(i+1)/2}H_{i-1}(u) e^{-u^2/2}=(2\pi)^{-i/2}H_{i-1}(u)\varphi(u), \forall i>0,$$
with $\displaystyle H_{j}(x)=(-1)^ne^{\frac{x^2}{2}} \frac{d^j}{dx^j}e^{\frac{-x^2}{2}}$ is the Hermite polynomial of degree $j$; and $\mu_i(S)$'s are the Minkowski functionals (or the Killing-Lipschitz curvatures) of $S$ (see \cite{MR2319516}). Note that $\mu_0(S) $ is the Euler characteristic of $S$, for example, it  is equal to number of connected components minus the number of holes inside when $ d=2$; and $\mu_d(S)$ is equal to $\lambda_d(S)$, the volume of $S$. 

The third method, the  Rice method,  is based on local maxima and  leads to the same approximation as in RHS of (\ref{EC1}). It gives also an upper bound. The first proof of validity is due to Piterbarg \cite{MR636766}. The expectation given in (\ref{EC1}) is proved to be a very accurate approximation when the domain $S$ is "nice" in the sense that it is a tamed and locally convex subset of $\R^d$ (see \cite[Theorem 14.3.3]{MR2319516}).  Note that in the case both apply,  the Euler Chacteristic method  gives  extra terms  with respect to the Double-sum method, and thus it is more accurate, see Azais and Mourareau  \cite{MR3949275}.
 
 In this paper, we are interested in the case $n\geq 2$. The motivation of this problem  comes from the statistical applications in neurology, for example, to determine whether the functional organization of the brain for language differs according to sex (see \cite{MR1747100}). In this application, $X_i(t)$ is the value of image $i$ at the location $t\in \R^d$ representing the intensity with respect to some actions. Here both the Double-sum method and Euler characteristic method are still useful. 
 
 By the Double-sum method, Debicki et al \citep{MR3178344,MR3385594} considered the one-dimensional processes and proved that
 $$\PP \left(\underset{t\in [0,T]}{\sup} \underset{1\leq i \leq n}{\min} X_i(t) >u \right)= H_{n,2}Tu\overline{\Phi}^n(u) (1+o(1)),$$
where $H_{n,2}$ is so-called the \textit{generalized Pickands constant} defined as
$$H_{n,2}=\underset{a\downarrow 0}{\lim} \frac{1}{a} \PP\left(\underset{k\geq 1}{\max} Z(ak) \leq 0\right),$$
with
$$Z(t)=\underset{1\leq i\leq n}{\min} \left( \sqrt{2}Y_i(t)-t^2+E_i\right),$$
here $Y_i$'s are independent copies of a centered Gaussian process $Y(t)$ with  covariance function $\textrm{Cov}(Y(t),Y(s))= |ts|,\; \forall t,s\geq 0$, and $E_i$'s are mutually independent unit mean exponential random variables being further independent of $Y_i$'s.  The expansion above must be understood, as in the rest of the paper,   as $u\to +\infty$.

Debicki et al  also considered non-stationary processes and mentioned that their result can be extended to  Gaussian fields but at the cost  of heavy notations. Note that their work is applied  to  a wider class of processes  than those considered here that are  smooth stationary ones.

By  the Euler characteristic method, Worsley and Friston \cite{MR1747100} considered the upper-triangular Toeplitz matrix $R$ defined as
\begin{equation}\label{matrix}
R= \begin{pmatrix}
\rho_0/b_0 & \rho_1/b_1 & \ldots & \rho_d/b_d\\
0 & \rho_0/b_0 & \ldots &\rho_{d-1}/b_{d-1}\\
\vdots &\vdots & \ddots&\vdots\\
0 & 0& \ldots & \rho_0/b_0 
\end{pmatrix},
\end{equation}

where $b_i=\Gamma((i+1)/2)/\Gamma(1/2)$ with  $\Gamma(.)$ the  Euler  gamma function, and the  $\rho_i$'s are the Euler characteristic densities as defined above.  They gave the heuristic argument that
\begin{equation}\label{ECH}
\PP(C_u\neq \emptyset)= \PP \left(\underset{t\in S}{\sup} \underset{1\leq i \leq n}{\min} X_i(t) \geq u \right) \approx \E (\mu_0(C_u))= (1,0,\ldots,0)R^n\mu(S),
\end{equation}
where $\mu(S)=(\mu_0(S)b_0,\mu_1(S)b_1,\ldots, \mu_d(S)b_d)$ is the column vector of the scaled  Minkowski functionals of $S$. However, to prove that $\E (\mu_0(C_u))$ is a good approximation is still an open question. For further discussion, see also \citep{MR2775212,MR2654766,MR1260300}.

Let us consider   the particular case of random processes (i.e. $d=1$). Then the matrix $R$ defined in (\ref{matrix}) becomes
$$R=\begin{pmatrix}
\rho_0 & \rho_1/b_1\\
0& \rho_0
\end{pmatrix}, \quad R^n=\begin{pmatrix}
\rho^n_0 & n\rho^{n-1}_0\rho_1/b_1\\
0& \rho^n_0
\end{pmatrix}.$$
Also note that in this case the domain $S$ is the interval $[0,T]$ with $\mu_0([0,T])=1$ and $\mu_1([0,T])=T$. Therefore, \textit{the validity of the Euler characteristic}, is equivalent to
$$\PP \left(\underset{t\in S}{\sup} \underset{1\leq i \leq n}{\min} X_i(t) \geq u \right) \approx \rho_0^n+n\rho_0^{n-1}\rho_1T=\overline{\Phi}^n(u)+n\overline{\Phi}^{n-1}(u)\varphi(u).T/\sqrt{2\pi}.$$
 That would give  and extra-term  with respect to  the  Double-sum method. And since
$$\overline{\Phi}(u)=\varphi(u)\left(\frac{1}{u}+o\left(\frac{1}{u}\right)\right),$$
 it could imply that
\begin{equation}\label{pick}
H_{n,2}=\frac{n}{\sqrt{2\pi}}.
\end{equation}
In a recent paper \cite{actata}, the equality in (\ref{pick}) has been proved to be true. The proof  exploits  the one-dimensional structure of the processes, and uses Rice formula to calculate the expected number of \textit{"up-crossing"}   of the level  $u$ while the other  processes are all greater than $u$. However, this idea seems hard to extend to  higher dimensions. 

In this paper, we consider the conjunction problem from another point of view. Our approach relies on a result of Azais and Wschebor \cite{MR3245991} describing the geometry of the excursion set. There they established a relation  between   the tail distribution of the maximum and the volume and the perimeter 
 of the index set.  In \cite{MR3473099}, this idea has been used to provide the asymptotic formula of the tail of the maximum corresponding to the coefficients of the volume of the $\epsilon$-neighborhood of non-locally convex index set. With the same spirit, we will give a  one-term asymptotic formula for the conjunction probability where the coefficient comes from the local geometry (or local volume) of the conjunction set (see Proposition \ref{volu}).  

Before stating the main result of this paper, we state  the technical assumptions on the considered  fields.

 \textbf{Assumption $A$: }Assume $\mathcal{X}$ be a random field defined on a ball $B\subset \R^d$ containing the domain $S$ such that $\mathcal{X}$ satisfies:
\begin{itemize}
\item[i.] $S$ is a stratified compact manifold (see \cite{MR2319516} for more detail) satisfying that it is the closure of its interior, and its boundary is the union of a finite number of $\mathcal{C}^2$ and $d-1$ dimensional compact domains.
\item[ii.] $\mathcal{X}$ is a stationary centered Gaussian field with unit variance and $\textnormal{Var}(X'(t))$ is the identity matrix.
\item[iii.] Almost surely the paths of $X(t)$ are of class $\mathcal{C}^3$.
\item[iv.] For all $s\neq t \in B$, the distribution of $(X(s),X(t),X'(s),X'(t))$ does not degenerate.
\item[v.] For all $t\in B$ and , $ \gamma$ in the unit sphere $\mathcal{S}^{d-1}$, the distribution of $(X(t),X'(t),X''(t)\gamma)$ does not degenerate.
\end{itemize}

Our main result is the following.

\begin{theorem}  \label{t:main}
Let $X_i(t), 1\leq i \leq n,$ be independent copies of a  Gaussian field $X$  satisfying Assumption (A). Then as $u$ tends to infinity,
\begin{align}
\displaystyle \PP \left(\underset{t\in S}{\max} \underset{1\leq i \leq n}{\min} X_i(t) \geq u \right)=&\displaystyle u^{d-n}\varphi^n(u)\left[ \frac{\lambda_d(S)}{(2\pi)^{d/2}}  \sum_{k_n=0}^d  \sum_{k_{n-1}=d-k_n}^d \ldots \sum_{k_{2}=(n-2)d-(k_n+k_{n-1}+\ldots+k_3)}^d\right. \notag 
\end{align}
\begin{equation} \label{mafomu}
\displaystyle \left. \frac{\omega_d}{\omega_{\sum_{i=2}^n k_i -(n-2)d} \prod_{i=2}^{n} \omega_{d-k_i}} \times \frac{d!}{ \left[ \sum_{i=2}^n k_i -(n-2)d\right]! .\prod_{i=2}^{n} (d-k_i)!} +o(1) \right],
\end{equation}

where $\omega_k$ stands for the volume of a $k$-dimensional unit ball.
\end{theorem}

The proof of the main theorem consists of two propositions  presented in Section 2.  We prove these propositions and consider some special examples in Sections 3 and 4. In Section 5, we compare our result with the corresponding term in the prediction given by Euler characteristic method. Although the  two formulas seem to be different but their values coincide. Therefore in some sense, our result confirms partially the validity of Euler characteristic method.

Throughout this paper, we will use the following  notation.

\begin{itemize}
\item[-] $\lambda_k(.)$ stands for the usual $k$-dimensional Lebesgue measure.
\item[-] $B(t,r)$ stands for the ball of radius $r$ at center $t$.

\item[-] For a $n$-dimensional vector $m=(m_1,\ldots,m_n)$ and a $n$-tuple of non-negative integers $r=(r_1,\ldots,r_n)$, the notations $m^r$ stands for
$$m^r=m_1^{r_1}m_2^{r_2} \ldots m_n^{r_n}.$$
\item[-] For a given set $S \subset \R^d$ and a positive constant $\epsilon$, the $\epsilon$- neighborhood of $S$,  denoted by $S^{+\epsilon}$, is defined as
$$S^{+\epsilon}=\{t\in \R^d:\; \mbox{dist}(t,S)\leq \epsilon\}.$$
\item[-] For a given set $S \subset \R^d$ and a small enough positive constant $\epsilon$, the set $S^{-\epsilon}$, is defined as
$$S^{-\epsilon}=\{t\in \R^d:\; B(t,\epsilon) \subset S\}.$$
\item[-] $\omega_k$ is the volume of a $k$-dimensional unit ball.
\item[-] $\|m\|$ is the $l_1$ norm of a vector. 
\end{itemize}
\section{Proof of the main theorem} 
The main result can be easily deduced from following two propositions.
\begin{proposition}\label{maint} Let $X_i(t), 1\leq i \leq n,$ be n  independent copies of a  Gaussian field $X$  satisfying Assumption (A). Assume that for a fixed point $t_1$ and small enough  $r_1,r_2,\ldots,r_n$, there exist  constants $k>0$ and $C_m$ such that
\begin{equation}\label{asymp2}
\lambda_{(n-1)d} \left((t_2,\ldots,t_n):\; \underset{1\leq i \leq n}{\cap} B(t_i,r_i) \neq \emptyset \right)=\sum_{\|m\|=k}C_mr^m.
\end{equation} 

Then, as $u$ tends to infinity,
$$\PP \left(\underset{t\in S}{\max} \underset{1\leq i \leq n}{\min} X_i(t) \geq u \right)=u^{nd-n-k}\varphi^n(u)\left(\frac{2^{k/2}\lambda_d(S)}{(2\pi)^{nd/2}}\sum_{\|m\|=k}C_m  \prod_{i=1}^n  \Gamma(1+m_i/2)  +o(1)\right).$$
\end{proposition}

\begin{proposition} \label{volu}
 For  $t_1$,  fixed point in the parameter  space and  for 
 $r_1,r_2,\ldots,r_n >0$ small enough, we have:
 \begin{align}
 &\displaystyle \lambda_{(n-1)d} \left((t_2,\ldots,t_n) \in \R^{d (n-1)} :\; \underset{1\leq i \leq n}{\cap} B(t_i,r_i) \neq \emptyset \right) \notag\\
 =& \displaystyle   \sum_{k_n=0}^d  \sum_{k_{n-1}=d-k_n}^d \ldots \sum_{k_{2}=(n-2)d-(k_n+k_{n-1}+\ldots+k_3)}^d r_1^{(n-1)d-\sum_{i=2}^n k_i} \times \notag \\
 &  \displaystyle  \prod_{i=2}^n \left( r_i^{k_i}\omega_{k_i} \right) \frac{\omega_d \omega_{(n-1)d-\sum_{i=2}^n k_i}}{\omega_{\sum_{i=2}^n k_i -(n-2)d} \prod_{i=2}^{n} \omega_{d-k_i}} \times \frac{d!}{ \left[ \sum_{i=2}^n k_i -(n-2)d\right]! .\prod_{i=2}^{n} (d-k_i)!}. \label{finalvo}
 \end{align}
\end{proposition}

Indeed, the explicit values of $k=(n-1)d$ and $C_m$'s are provided in Proposition \ref{volu}. Then substituting them in Proposition \ref{maint}, we get the asymptotic formula for the conjunction probability as 
\begin{align*}
&\displaystyle \PP \left(\underset{t\in S}{\max} \underset{1\leq i \leq n}{\min} X_i(t) \geq u \right)\\
=&\displaystyle u^{d-n}\varphi^n(u)\left[ \frac{2^{(n-1)d/2}\lambda_d(S)}{(2\pi)^{nd/2}}  \sum_{k_n=0}^d  \sum_{k_{n-1}=d-k_n}^d \ldots \sum_{k_{2}=(n-2)d-(k_n+k_{n-1}+\ldots+k_3)}^d\right. \\
&\displaystyle  \Gamma \left( \frac{(n-1)d-\sum_{i=2}^n k_i}{2}+1\right) \omega_{(n-1)d-\sum_{i=2}^n k_i} \prod_{i=2}^n \Gamma \left( \frac{k_i}{2}+1\right) \omega_{k_i} \times\\
&\displaystyle \left. \frac{\omega_d}{\omega_{\sum_{i=2}^n k_i -(n-2)d} \prod_{i=2}^{n} \omega_{d-k_i}} \times \frac{d!}{ \left[ \sum_{i=2}^n k_i -(n-2)d\right]! .\prod_{i=2}^{n} (d-k_i)!} +o(1) \right].
\end{align*}

Using the fact that $\displaystyle \omega_k=\frac{\pi^{k/2}}{\Gamma(1+k/2)},$ the proof is completed.
 \section{Proof of Proposition \ref{maint}} 
 In the proof of Proposition \ref{maint}, we need  following two lemmas. The first lemma describes the role of (\ref{asymp2}). 
 \begin{lemma} Assume that there exist constants $k$ and $C_m$'s such that the equality in (\ref{asymp2}) holds. Then for  small enough  $r_1,r_2,\ldots,r_n,$ we have
 
 \begin{equation}\label{asymp}
\lambda_{nd} \left((t_1,\ldots,t_n):\; t_1 \in S^{+r_1},  \,  \underset{1\leq i \leq n}{\cap} B(t_i,r_i) \neq \emptyset \right)\leq \lambda_d (S)\sum_{\|m\|=k}C_mr^m +O(\sum_{\|m\|=k+1}r^m),
\end{equation} 
and
\begin{equation}\label{asymp1}
\lambda_{nd} \left((t_1,\ldots,t_n):\; t_1 \in S^{-r_1}, \forall i \, ; \,  \underset{1\leq i \leq n}{\cap} B(t_i,r_i) \neq \emptyset \right)\geq \lambda_d(S) \sum_{\|m\|=k}C_mr^m +O(\sum_{\|m\|=k+1}r^m).
\end{equation}
 \end{lemma}
 
 \begin{proof} It is clear that
 \begin{align*}
 & \lambda_{nd} \left((t_1,\ldots,t_n):\; t_1 \in S^{+r_1},  \,  \underset{1\leq i \leq n}{\cap} B(t_i,r_i) \neq \emptyset \right) \\ 
 =& \displaystyle \int \mathbb{I}_{\{t_1 \in S^{+r_1}\}} dt_1 \int \mathbb{I}_{\{\underset{1\leq i \leq n}{\cap} B(t_i,r_i) \neq \emptyset\}} dt_2\ldots dt_n  = \displaystyle \lambda_d(S^{+r_1}) \sum_{\|m\|=k}C_mr^m.
 \end{align*} 
 Then (\ref{asymp}) follows from the facts that 
 $$\lambda_d(S^{+r_1}) \leq \lambda_d(S) + \lambda_d(\partial S^{+r_1}), $$
 where $\partial S$ stands for the boundary of $S$ that is the union of a finite number of $\mathcal{C}^2$ and $d-1$ dimensional compact domains $T_j$'s; and for each domain $T_j$, one has the tube formula that (see \cite{MR2024928})
 $$\lambda_d(T_j^{+r_1})=O(r_1).$$
 
 Similarly, to prove (\ref{asymp1}), we have
 \begin{align*}
 & \lambda_{nd} \left((t_1,\ldots,t_n):\; t_1 \in S^{-r_1},  \,  \underset{1\leq i \leq n}{\cap} B(t_i,r_i) \neq \emptyset \right) \\ 
 =& \displaystyle \lambda_d(S^{-r_1}) \sum_{\|m\|=k}C_mr^m  \geq  \left(\lambda_d(S)- \lambda_d(\partial S^{+r_1})\right) \sum_{\|m\|=k}C_mr^m,  
 \end{align*} 
 and the proof follows easily.
 \end{proof}

The second lemma is due to Azais and Wschebor \cite{MR3245991}.
\begin{lemma}\label{leaz}
Let $X$ be a random field  satisfying Assumption (A) and $\alpha$ be a  real number,  $0<\alpha <1$. Then the  following event  
occurs with high probability, in the sense that there exist two constants $C,c>1$ such that its probability is at least equal to $1-Ce^{-cu^2/2}$.  The event is described by :

\begin{itemize}
\item[+] The field has only one local maximum point $t_0\in \overset{\circ}{B}$ with value $X(t_0) \in [u,u+1]$,
\item[+] and  the  excursion set 
$$K_{u} :=\{s\in B: \, X(s) \geq u\}$$
consists of only one connected component, and moreover, 
$$ B(t_0,\underline{r})  \subset K_{u}  \subset B(t_0,\overline{r}),$$
where $\displaystyle    \underline{r}=\sqrt{2\frac{X(t_0)-u}{X(t_0)+u^{\alpha}}}$ and $\displaystyle    \overline{r}=\sqrt{2\frac{X(t_0)-u}{X(t_0)-u^{\alpha}}}$.
\end{itemize}

\end{lemma}

From Lemma \ref{leaz}, for each $i=1,\ldots ,n$, with high probability,  the following event $H_i$  occurs: there exists only one local maximum of $X_i(t)$ at the location $t_i\in \overset{\circ}{B_i}=\overset{\circ}{B}$ with value in $[u,u+1]$, and  the corresponding excursion set 
$K_{u,i} :=\{s\in B_i: \, X_i(s) \geq u\}$ satisfies that 
$$ B(t,\underline{r_i})  \subset K_{u,i}  \subset B(t_i,\overline{r_i}),$$
where $\displaystyle    \underline{r_i}=\sqrt{2\frac{X_i(t_i)-u}{X_i(t_i)+u^{\alpha}}}$ and $\displaystyle    \overline{r_i}=\sqrt{2\frac{X_i(t_i)-u}{X_i(t_i)-u^{\alpha}}}$.

Moreover, if for some $i\in \{1,\ldots,n\},$ the complement of the above event $H_i$ occurs, then
\begin{align*}
&\PP \left(\bar{H_i}\cap \{\underset{t\in S}{\sup} \underset{1\leq k \leq n}{\min} X_k(t) \geq u \}\right) \\
\leq & \PP \left(\bar{H_i}\right) \PP\left(\underset{t\in S}{\sup} \underset{1\leq k \leq n, k\neq i}{\min} X_k(t) \geq u \right)\\
\leq & (const) e^{-cu^2/2} \left( \lambda_d(S) u^{d-1}\varphi(u) \right)^{n-1}=o(u^{nd-n-k}\varphi^n(u)).
\end{align*}
Here we use the fact that (see \cite{MR1361884})
$$ \PP (\underset{t\in S}{\max} X_i(t) \geq u ) \leq  (const) \lambda_d(S) u^{d-1}\varphi(u).$$
Therefore, from the fact that
\begin{align*}
\PP \left(\underset{t\in S}{\sup} \underset{1\leq i \leq n}{\min} X_i(t) \geq u \right) &=\PP(\exists t\in S:\, t\in K_{u,i} \forall i=1,\ldots,n)\\
&= \PP(S\cap K_{u,1}\cap\ldots\cap K_{u,n}\neq \emptyset ), 
\end{align*}
 we obtain the upper bound
$$\PP \left(\underset{t\in S}{\sup} \underset{1\leq i \leq n}{\min} X_i(t) \geq u \right)\leq \PP(S\cap B(t_1,\overline{r_1})\cap\ldots\cap B(t_n,\overline{r_n})\neq \emptyset) +o(u^{nd-n-k}\varphi^n(u)),$$
and the lower bound
$$\PP \left(\underset{t\in S}{\sup} \underset{1\leq i \leq n}{\min} X_i(t) \geq u \right)\geq \PP(S\cap B(t_1,\underline{r_1})\cap\ldots\cap B(t_n,\underline{r_n})\neq \emptyset)+o(u^{nd-n-k}\varphi^n(u)).$$

$\bullet$ At first, we deal with the upper bound. By  the Markov inequality, it is at most equal to
 \begin{align*}
 &\PP( \exists t=(t_1,\ldots,t_n)\in B^{\otimes n}: \; \forall i=1,\ldots,n, \, X_i(t) \, \mbox{has a local maximum at} \, t_i,  \\
 & \quad \quad X_i(t_i)\in [u,u+1], t_1 \in S^{+\overline{r_1}} \, \mbox{and} \,  \underset{1\leq i \leq n}{\cap} B(t_i,\overline{r_i}) \neq \emptyset)+o(u^{nd-n-k}\varphi^n(u))\\
 \leq & \E (\mbox{card} \{t=(t_1,\ldots,t_n)\in B^{\otimes n}: \; \forall i=1,\ldots,n, \, X_i(t) \, \mbox{has a local maximum at} \, t_i,  \\
 & \quad \quad  X_i(t_i)\in [u,u+1], t_1 \in S^{+\overline{r_1}} \, \mbox{and} \, \underset{1\leq i \leq n}{\cap} B(t_i,\overline{r_i}) \neq \emptyset \} )+o(u^{nd-n-k}\varphi^n(u)),
\end{align*}
where $B^{\otimes n}$ stands for the Cartesian product set $B\times \ldots\times B$.

By  Rice formula  applied to the vector-valued Gaussian field $Z(t)=(X'_1(t_1),\ldots, X'_n(t_n))$ with $t=(t_1,\ldots,t_n)\in B^{\otimes n}$, the above expectation is equal to
\begin{align*}
E=&\displaystyle \int_{[u,u+1]^{\otimes n}}du_1\ldots du_n\int_{B^{\otimes n}} dt \;p_{X_1(t_1),\ldots,X_n(t_n),X'_1(t_1),...,X'_n(t_n)}(u_1,\ldots,u_n,0,\ldots,0)\\
& \times\E\left(\left|\prod_{i=1}^n \det\left( X^{''}_i(t_i) \right)\mathbb{I}_{\{X^{''}_i(t_i)\preceq 0\}}\right| \mathbb{I}_{\{t_1\in S^{+\overline{r_1}}\} } \mathbb{I}_{\{ \underset{1\leq i \leq n}{\cap} B(t_i,\overline{r_i}) \neq \emptyset \}}\mid X_i(t_i)=u_i,X'_i(t_i)=0 \, \forall i\right),
\end{align*}
where $p_{X_1(t_1),\ldots,X_n(t_n),X'_1(t_1),...,X'_n(t_n)}(.)$ is the joint density function of the random vector $$(X_1(t_1),\ldots,X_n(t_n),X'_1(t_1),...,X'_n(t_n)).$$

Using (\ref{asymp}) and the fact that the fields $X_i$'s are independent and $X'_i(t_i)$ is independent to $X_i(t_i)$ and $X^{''}_i(t_i)$, we have
\begin{align}
 E=\frac{\lambda_d(S)}{(2\pi)^{nd/2}}\int_{[u,u+1]^{\otimes n}} &\prod_{i=1}^n\E\left(\left|  \det \left( X^{''}_i(t_i) \right) \mathbb{I}_{\{X^{''}_i(t_i)\preceq 0\}}\right|\mid X_i(t_i)=u_i\right)\varphi(u_i) \notag \\
 &\quad \times\left(\sum_{\|m\|=k}C_m\overline{r}^m +O(\sum_{\|m\|=k+1}\overline{r}^m )\right)  du_1\ldots du_n. \label{aldo}
\end{align}
Note that under the condition $X_i(t_i)=u_i$ then $\overline{r_i}$ is no more random and is equal to
$$\overline{r_i}=\sqrt{2\frac{u_i-u}{u_i-u^{\alpha}}}.$$ 
Using the fact that (see \cite{MR1965978})
$$\E\left(|\det(X'_i(t))|\mathbb{I}_{\{X_i''(t)\preceq 0\}}\mid X_i(t)=u_i,X'_i(t)=0\right)=u_i^d+ O\left(u_i^{d-2}\right)\; \textnormal{as} \; u_i\rightarrow \infty ,$$
then
\begin{align*}
& \int_u^{u+1} \overline{r_i}^{m_i}  \E\left(\left| \det \left(X^{''}_i(t_i)\right)\mathbb{I}_{\{X^{''}_i(t_i)\preceq 0\}}\right| \mid X_i(t_i)=u_i,X'_i(t_i)=0 \, \right)\varphi(u_i)du_i\\
\simeq & \int_u^{u+1} u_i^d \left(2\frac{u_i-u}{u_i-u^{\alpha}}\right)^{m_i/2} \varphi(u_i)du_i.
\end{align*}
By the change of variable $u_i=u+x/u$, the above integral is equal to
\begin{align*}
 & \int_u^{u+1} (u+x/u)^d \left(\frac{2x/u}{u+x/u-u^{\alpha}}\right)^{m_i/2} \varphi(u+x/u)dx/u\\
 \simeq & 2^{m_i/2}u^{d-(m_i+1)}\varphi(u)\int_0^ux^{m_i/2}e^{-x}dx \simeq 2^{m_i/2}u^{d-(m_i+1)}\varphi(u) \Gamma(1+m_i/2).
\end{align*}
Therefore, for each vector $m=(m_1,\ldots,m_n)$ with norm $k$,
\begin{align*}
 &\int_{[u,u+1]^{\otimes n}} \prod_{i=1}^n\E\left(\left| \det\left( X^{''}_i(t_i) \right)\mathbb{I}_{\{X^{''}_i(t_i)\preceq 0\}}\right|\mid X_i(t_i)=u_i\right)\varphi(u_i) \overline{r}^m du_1\ldots du_n\\
 =& \prod_{i=1}^n \int_u^{u+1} \overline{r_i}^{m_i}\E\left(\left|  \det\left( X^{''}_i(t_i) \right)\mathbb{I}_{\{X^{''}_i(t_i)\preceq 0\}}\right|\mid X_i(t_i)=u_i\right)\varphi(u_i) du_i\\
 \simeq & \prod_{i=1}^n  2^{m_i/2}u^{d-(m_i+1)}\varphi(u) \Gamma(1+m_i/2)\\
 = &2^{k/2}u^{nd-n-k}\varphi^n(u) \prod_{i=1}^n  \Gamma(1+j_i/2).
\end{align*}
Hence, substituting into (\ref{aldo}),
$$E=u^{nd-n-k}\varphi^n(u)\left(\frac{2^{k/2}\lambda_d(S)}{(2\pi)^{nd/2}}\sum_{\|m\|=k}C_m  \prod_{i=1}^n  \Gamma(1+m_i/2)  +o(1)\right).$$

$\bullet$ For the lower bound, recall that
$$S^{-\underline{r_1}}=\{t\in S: \; B(t,\underline{r_1}) \subset S\}.$$
Then the lower bound is at least equal to
\begin{align*}
& \PP\left(\exists t \in B^{\otimes n}:\, t_1\in S^{-\underline{r_1}},\, \forall i: \,X_i(t) \, \mbox{has a local maximum at} \, t_i , X_i(t_i)\in [u,u+1],\right.\\
& \quad \left. \, \mbox{and} \,  \underset{1\leq i \leq n}{\cap} B(t_i,\underline{r_i}) \neq \emptyset  \right)+o(u^{nd-n-k}\varphi^n(u))\\
=& \PP(M_{\underline{r}} \geq 1)+o(u^{nd-n-k}\varphi^n(u))\\
\geq & \E(M_{\underline{r}}) - \E(M_{\underline{r}}(M_{\underline{r}}-1))/2 +o(u^{nd-n-k}\varphi^n(u))\\
\end{align*}
where
\begin{align*}
M_{\underline{r}}=&\mbox{card} \left\{t \in B^{\otimes n}:\, t_1\in S^{-\underline{r_1}},\,  \, \underset{1\leq i \leq n}{\cap} B(t_i,\underline{r_i}) \neq \emptyset,\right.\\
& \quad \quad \quad \quad  \, \mbox{and} \, \forall i: \;  X_i(t) \, \mbox{has a local maximum at} \, t_i , X_i(t_i)\in [u,u+1] \, \}.
\end{align*}
It is clear that
$$M_{\underline{r}}\leq M=\mbox{card} \{t=(t_1,\ldots,t_n)\in B^{\otimes n}: \;  X_i(t) \, \mbox{has a local maximum at} \, t_i, X_i(t_i)\geq u, \forall i\}.$$
Then applying  the Rice formula and using the independent property of the given fields, we have
\begin{align*}
&\E(M_{\underline{r}}(M_{\underline{r}}-1)) \leq \E (M.(M-1))\\
=& \int_{[u,\infty )^{\otimes n}\times [u,\infty )^{\otimes n}}dydz \int_{B^{\otimes n}\times B^{\otimes n}} dtds \\
&\times   \E\left(\left|\prod_{i=1}^n \det \left(X^{''}_i(t_i)\right)\mathbb{I}_{\{X^{''}_i(t_i)\preceq 0\}}  \det \left(X^{''}_i(s_i)\right)\mathbb{I}_{\{X^{''}_i(s_i)\preceq 0\}}\right| \right. \\
& \quad\quad\quad\quad\quad\quad\quad\quad\quad\quad\quad\quad\quad\left. \mid X_i(t_i)=y_i,X'_i(t_i)=0, X_i(s_i)=z_i,X'_i(s_i)=0,\, \forall i\right)\\
& \times p_{X_1(t_1),\ldots,X_n(t_n),X'_1(t_1),...,X'_n(t_n),X_1(s_1),\ldots,X_n(s_n),X'_1(s_1),...,X'_n(s_n)}(y,0,z, 0)\\
=& \prod_{i=1}^n \int_{[u,\infty )\times [u,\infty )}dy_idz_i  \int_{B\times B} dt_ids_i \times p_{X_i(t_i),X'_i(t_i),X_i(s_i),X'_i(s_i)}(y_i,0,z_i,0)\\
& \times \E\left(\left|\det \left(X^{''}_i(t_i)\right)\mathbb{I}_{\{X^{''}_i(t_i)\preceq 0\}}  \det \left(X^{''}_i(s_i)\right)\mathbb{I}_{\{X^{''}_i(s_i)\preceq 0\}}\right| \right. \\
& \quad\quad\quad\quad\quad\quad\quad\quad\quad\quad\quad\quad\quad\left. \mid X_i(t_i)=y_i,X'_i(t_i)=0, X_i(s_i)=z_i,X'_i(s_i)=0\right)\\
=& \prod_{i=1}^n \E (M_i.(M_i-1)),
\end{align*}
where
$$M_i=\textnormal{card}\left\{ t_i\in \overset{\circ}{B_i}: \; X_i(.)\textnormal{ \;has\; a\; local\; maximum\; at}\; t_i, \; X(t)\geq u\right\}.$$
In \cite{MR1965978}, it is proved  that there exist two constants $C,c>1$ such that 
$$\E (M_i.(M_i-1)) \leq Ce^{-cu^2/2}.$$
Hence we have
$$\E (M.(M-1))=o(u^{nd-n-k}\varphi^n(u)).$$
 The calculation of the expectation $\E(\mathcal{M}_{\underline{r}})$ is similar as in the upper bound part and we obtain the same asymptotic formula. Then the result follows.

\section{Proof of Proposition \ref{volu} and some examples}
 Before giving the  proof, we would like to consider three interesting examples where we can give a much simpler formula for the volume in (\ref{asymp2}), and therefore a simpler asymptotic formula for the conjunction probability. We hope that through these examples, the readers can get the intuition about the basic ideas of the detailed proof. 
\subsection{First example: $n=2$ }
This example corresponds to the practical application mentioned in the introduction.  It is clear that
$$\{t_2:\; B(t_1,r_1) \cap B(t_2,r_2) \neq \emptyset \} =B(t_1,r_1+r_2).$$
Therefore
\begin{align*}
& \displaystyle \lambda_{d} \left(t_2:\; B(t_1,r_1) \cap B(t_2,r_2) \neq \emptyset \right) =\lambda_d \left( B(t_1,r_1+r_2) \right)\\
=& \displaystyle \frac{\pi^{d/2}}{\Gamma (1+d/2)} (r_1+r_2)^d=\frac{\pi^{d/2}}{\Gamma (1+d/2)} \sum_{j=0}^d \binom{d}{j} r_1 ^j r_2^{d-j},
\end{align*}
here we use again the fact that the volume of a $d$-dimensional unit ball is $\pi^{d/2}/\Gamma (1+d/2)$.

Then we have an immediate consequence of Proposition \ref{maint} as follows.
\begin{corollary}
 Consider  $X_i(t), 1\leq i \leq 2,$  two independent copies of a  Gaussian field $X$  satisfying Assumption (A). Then as $u$ tends to infinity,
\begin{align}\label{hidim}
 &\PP\left(\underset{t\in S}{\max} \{\min (X_1(t),X_2(t))\}  \geq u \right)\notag \\
 = &\frac{u^{d-2}\varphi^2(u)\lambda_d(S)}{(2\pi)^{d/2}\Gamma(1+d/2)} \left(\sum_{j=0}^d \binom{d}{j}\Gamma(1+j/2)  \Gamma(1+(d-j)/2) +o(1)\right).
\end{align}
\end{corollary}
\begin{proof} We substitute the following parameters in the statement of the main theorem
$$k=d, \quad m=(j,d-j), \; \mbox{and} \;  C_m=\frac{\pi^{d/2}\binom{d}{j}}{\Gamma (1+d/2)}.$$
\end{proof}

\textbf{Remark.} Let us now consider the estimation given by Euler characteristic method. It is clear that (\ref{ECH}) becomes
$$(1,0,\ldots,0)R^2\mu(S).$$
Here the term corresponding to $\mu_d(S)$ (or $\lambda_d(S)$) is
$$\lambda_d(S)b_d\sum_{i=0}^d\frac{\rho_i}{b_i}\frac{\rho_{d-i}}{b_{d-i}}.$$
From the definition of the Euler characteristic densities $\rho_i$'s, this term is equivalent to
$$\frac{\lambda_d(S)u^{d-2}\varphi^2(u)}{(2\pi)^{d/2}}\Gamma((d+1)/2)\Gamma(1/2)\sum_{i=0}^d\frac{1}{\Gamma((i+1)/2)\Gamma((d-i+1)/2)}. $$
Comparing with the asymptotic formula given in (\ref{hidim}), it is surprising to see that 

\begin{align*}
&\Gamma((d+1)/2)\Gamma(1/2)\sum_{i=0}^d\frac{1}{\Gamma((i+1)/2)\Gamma((d-i+1)/2)}\\
=& \frac{1}{\Gamma(1+d/2)} \sum_{i=0}^d \binom{d}{i}\Gamma(1+i/2)  \Gamma(1+(d-i)/2).
\end{align*}
Indeed, we will prove that for every $i=0,\ldots,d,$
\begin{equation}\label{iparti}
\frac{\Gamma((d+1)/2)\Gamma(1/2)}{\Gamma((i+1)/2)\Gamma((d-i+1)/2)}= \frac{1}{\Gamma(1+d/2)} \binom{d}{i}\Gamma(1+i/2)  \Gamma(1+(d-i)/2).
\end{equation}
The equality (\ref{iparti}) is equivalent to
\begin{align*}
 &\Gamma(d/2+1/2) \Gamma(d/2+1) \Gamma(1/2)\\
=&\displaystyle \binom{d}{i} \Gamma(i/2+1/2)\Gamma(i/2+1)\Gamma((d-i)/2+1/2)  \Gamma((d-i)/2+1),
\end{align*}
that is true from Legendre duplication formula
\begin{equation}\label{legendre}
\Gamma(z)\Gamma\left( z+\frac{1}{2}\right)=2^{1-2z}\sqrt{\pi}\Gamma(2z),
\end{equation}
and  from  $\Gamma(n+1)=n!, \quad \Gamma(1/2)=\sqrt{\pi}.$ 
\subsection{Second example: $d=1$}
In this subsection, we would like to revisit the conjunction probability of  stationary centered Gaussian processes. Although that the corresponding result given in \cite{actata} is more powerful and more informative than  the asymptotic formula given in Theorem \ref{t:main} , it would be nice to reprove that 
$$H_{n,2}=\frac{n}{\sqrt{2\pi}}.$$
The affirmative answer is deduced by the following lemma.
\begin{lemma}\label{lem1dim}  For a fixed point $t_1$ on the real axis
 and small enough fixed radii $r_1,r_2,\ldots,r_n$, we have
\begin{equation}\label{1dimn}
\lambda_{n-1} \left((t_2,\ldots,t_n) \in \R^{n-1} :\; \underset{1\leq i \leq n}{\cap} B(t_i,r_i) \neq \emptyset \right)=2^{n-1}\sum_{i=1}^n \left( \prod_{j\neq i}  r_j \right).
\end{equation}
\end{lemma}
\begin{proof}
We will prove by induction on $n$.

$\bullet$ For $n=2$, it is obvious as in the above subsection.

$\bullet$ Assume that the statement is true from $2$ to $n-1$.  

$\bullet$ For $n$-tuple $(t_1,t_2,\ldots,t_n)$, we would like to calculate the volume as  the following integral

\begin{align}
& \displaystyle \int_{\R^{n-1}} \mathbb{I}_{\{\underset{1\leq i \leq n}{\cap} B(t_i,r_i) \neq \emptyset \}} dt_2\ldots dt_{n} \notag \\
= & \displaystyle\int_{\R^{n-2}} \mathbb{I}_{\{\underset{1\leq i \leq n-1}{\cap} B(t_i,r_i) \neq \emptyset \}} dt_2\ldots dt_{n-1} \int_{\R}\mathbb{I}_{\{B(t_n,r_n) \cap \left( \underset{1\leq i \leq n-1}{\cap} B(t_i,r_i) \right) \neq \emptyset \}} dt_n \label{didim}.
\end{align}
Again by induction, it is clear that if the intersection $\displaystyle \underset{1\leq i \leq n-1}{\cap} B(t_i,r_i)$ is non-empty, it is an interval. Therefore
$$\int_{\R}\mathbb{I}_{\{B(t_n,r_n) \cap \left( \underset{1\leq i \leq n-1}{\cap} B(t_i,r_i) \right) \neq \emptyset \}} dt_n=\lambda_1 \left( \underset{1\leq i \leq n-1}{\cap} B(t_i,r_i) \right) +2r_n, $$
and substitute in (\ref{didim}), the considering volume is equal to
\begin{equation}\label{didim1}
\int_{\R^{n-2}} \mathbb{I}_{\{\underset{1\leq i \leq n-1}{\cap} B(t_i,r_i) \neq \emptyset \}} \left( \lambda_1 \left( \underset{1\leq i \leq n-1}{\cap} B(t_i,r_i) \right) +2r_n \right) dt_2\ldots dt_{n-1}.
\end{equation}
By inductive hyphothesis,
$$2r_n \int_{\R^{n-2}} \mathbb{I}_{\{\underset{1\leq i \leq n-1}{\cap} B(t_i,r_i) \neq \emptyset \}}\ldots dt_{n-1}=2r_n.2^{n-2}\sum_{i=1}^{n-1} \left( \prod_{1\leq j \leq n-1, \; j\neq i}  r_j \right).$$
For the rest term in the integral (\ref{didim1}), let us introduce a new variable $y$ corresponding to the point in the intersection, and we have
\begin{align*}
& \displaystyle \int_{\R^{n-2}} \mathbb{I}_{\{\underset{1\leq i \leq n-1}{\cap} B(t_i,r_i) \neq \emptyset \}}  \lambda_1 \left( \underset{1\leq i \leq n-1}{\cap} B(t_i,r_i) \right)  dt_2\ldots dt_{n-1}\\
=& \displaystyle \int_{\R^{n-2}} \int_{\R} \mathbb{I}_{\{\underset{1\leq i \leq n-1}{\cap} B(t_i,r_i) \neq \emptyset \}}   \mathbb{I}_{\{y \in \underset{1\leq i \leq n-1}{\cap} B(t_i,r_i) \}}   dt_2\ldots dt_{n-1} dy \\
= & \displaystyle \int_{B(t_1,r_1)}dy \left(\int_{\R^{n-2}}  \mathbb{I}_{\{y \in \underset{1\leq i \leq n-1}{\cap} B(t_i,r_i) \}}   dt_2\ldots dt_{n-1} \right)\\
= & \displaystyle \int_{B(t_1,r_1)}dy \left( \int_{B(y,r_2)}dy_2 \ldots \int_{B(y,r_{n-1})} dt_{n-1}\right)\\
= & \displaystyle \int_{B(t_1,r_1)} \left(\prod_{i=2}^{n-1} (2r_i)  \right)dy=2^{n-1}\prod_{i=1}^{n-1} r_i,
\end{align*}
where the equality in the third line follows from Fubini theorem. The result follows easily.
\end{proof}

Applying Proposition \ref{maint} in this case with respect to $k=n-1$, $m$ is $n$- dimensional vector with $n-1$ unit entries and only one zero entry and $C_m=2^{n-1}$, we have

\begin{corollary}
 Let $X_i(t), 1\leq i \leq n,$ be the independent copies of a  Gaussian process $X$  satisfying Assumption (A). Then as $u$ tends to infinity,
$$\PP \left(\underset{t\in [0,T]}{\max} \underset{1\leq i \leq n}{\min} X_i(t) \geq u \right)=u^{-(n-1)}\varphi^n(u)\left(\frac{nT}{\sqrt{2\pi}} +o(1)\right).$$
\end{corollary}

\subsection{Third example: $d=2$ }

\begin{lemma} For a fixed point $t_1$ in the plane and small enough fixed radii $r_1,r_2,\ldots,r_n$, we have
\begin{align*}
&\lambda_{2(n-1)} \left((t_2,\ldots, t_n) \in \R^{2(n-1)} :\; \underset{1\leq i \leq n}{\cap} B(t_i,r_i) \neq \emptyset \right)\\
=&\pi^{n-1)} \sum_{i=1}^n \left(\prod_{j\neq i} r_j^2 \right)+2\pi^{n-1} \sum_{1\leq i<j\leq n} \left(r_i r_j\prod_{k\neq i,j} r_k^2 \right) .
\end{align*}
\end{lemma}

\begin{proof} We will prove by induction on $n$.  

$\bullet$ Case $n=2$ has been considered in Subsection 4.1.

$\bullet$ For general $n\geq 3$, we have
\begin{align*} 
 &\lambda_{2(n-1)} \left((t_2,t_3,\ldots,t_n) \in \R^{2(n-1)} :\; \underset{1\leq i \leq n}{\cap} B(t_i,r_i) \neq \emptyset \right) \\
 =& \displaystyle \int_{\R^{2(n-2)}} dt_2 \ldots dt_{n-1} \left[ \mathbb{I}_{\{\underset{1\leq i \leq n-1}{\cap} B(t_i,r_i)\neq \emptyset \}} \int_{\R^2} \mathbb{I}_{\{ B(t_n,r_n) \cap \left( \underset{1\leq i \leq n-1}{\cap} B(t_i,r_i) \right) \neq \emptyset \}}  dt_n\right] \\
= & \displaystyle \int_{\R^{2(n-2)}} dt_2 \ldots dt_{n-1} \left[ \mathbb{I}_{\{\underset{1\leq i \leq n-1}{\cap} B(t_i,r_i)\neq \emptyset \}} \lambda_2 \left(  \left( \underset{1\leq i \leq n-1}{\cap} B(t_i,r_i) \right)^{+r_n}  \right)\right] .\\
\end{align*}
Since the intersection $ \left( \underset{1\leq i \leq n-1}{\cap} B(t_i,r_i) \right)$ is a convex set then from the Steiner formula,
$$ \lambda_2 \left(  \left( \underset{1\leq i \leq n-1}{\cap} B(t_i,r_i) \right)^{+r_n}  \right)=\pi r_n^2 +r_n . \mbox{peri}\left( \underset{1\leq i \leq n-1}{\cap} B(t_i,r_i) \right) + \lambda_2 \left( \underset{1\leq i \leq n-1}{\cap} B(t_i,r_i) \right), $$
where \textit{peri(.)} stands for the perimeter of the set.

Therefore, the considering volume is equal to
\begin{align} \label{genen2}
&\displaystyle \int_{\R^{2(n-2)}} dt_2 \ldots dt_{n-1} \mathbb{I}_{\{\underset{1\leq i \leq n-1}{\cap} B(t_i,r_i)\neq \emptyset \}} &\\
&\quad\quad\quad\quad\quad\quad\quad \left[  \pi r_n^2 +r_n . \mbox{peri}\left( \underset{1\leq i \leq n-1}{\cap} B(t_i,r_i)\right) + \lambda_2 \left( \underset{1\leq i \leq n-1}{\cap} B(t_i,r_i) \right)  \right].& \notag
\end{align}

$-$ For the first term in (\ref{genen2}), by the inductive hypothesis,
\begin{align*}
&\pi r_n^2 \int_{\R^{2(n-2)}} dt_2 \ldots dt_{n-1} \mathbb{I}_{\{\underset{1\leq i \leq n-1}{\cap} B(t_i,r_i)\neq \emptyset \}} \\
=& \pi r_n^2  \left[ \pi^{n-2} \sum_{i=1}^{n-1} \left(\prod_{j\neq i} r_j^2 \right)+2\pi^{n-2} \sum_{1\leq i<j\leq n-1} \left(r_i r_j\prod_{k\neq i,j} r_k^2 \right) \right].
\end{align*}

$-$ For the third term in (\ref{genen2}), we introduce a new variable $y$ corresponding to the point in the intersection, and we use Fubini theorem to obtain that
\begin{align*}
&\displaystyle \int_{\R^{2(n-2)}} dt_2 \ldots dt_{n-1} \lambda_2\left(\underset{1\leq i \leq n-1}{\cap} B(t_i,r_i) \right)= \int_{\R^{2(n-2)}} dt_2 \ldots dt_{n-1}  \left[ \int_{\underset{1\leq i \leq n-1}{\cap} B(t_i,r_i) } dy\right] \\
=& \displaystyle \int_{B(t_1,r_1)} dy \left[ \int_{\R^{2(n-2)}} dt_2 \ldots dt_{n-1} \prod_{i=2}^{n-1} \mathbb{I}_{\{t_i \in B(y,r_i)\}} \right] = \pi^{n-1} \prod_{i=1}^{n-1}r_i^2.\\
\end{align*}

 $-$ For the second term in (\ref{genen2}), let us denote $S(t,r)$ the circle with radius $r$ at center point $t$, i.e. the boundary of the disk $B(t,r)$. It is clear that the perimeter of the intersection $\left( \underset{1\leq i \leq n-1}{\cap} B(t_i,r_i)\right)$ is the sum of the lengths of the arcs on each circle $S(t_i,r_i),  i=\overline{1,n-1}$. For the first kind with respect to the arc on $S(t_1,r_1)$, again by Fubini theorem, we have
\begin{align*}
&\displaystyle r_n\int_{\R^{2(n-2)}} dt_2 \ldots dt_{n-1} \left[ \int_{S(t_1,r_1)} \mathbb{I}_{\{y\in \underset{2\leq i \leq n-1}{\cap} B(t_i,r_i)\}} dy\right]\\
=& \displaystyle r_n\int_{S(t_1,r_1)} dy \left[ \int_{\R^{2(n-2)}} dt_2 \ldots dt_{n-1} \prod_{i=2}^{n-1} \mathbb{I}_{\{t_i \in B(y,r_i)\}} \right] =2\pi^{n-1}r_1r_n \prod_{i=2}^{n-1}r_i^2.
\end{align*}

For the second kind with respect to the arc on $S(t_i,r_i)$ with $i=2,\ldots,n-1$. Without loss of generality, we consider the arc on $S(t_2,r_2)$. We have
\begin{align*}
&\displaystyle r_n\int_{\R^{2(n-2)}} dt_2 \ldots dt_{n-1} \left[ \int_{S(t_2,r_2)} \mathbb{I}_{\{y\in   \underset{ i\neq 2}{\cap} B(t_i,r_i)\} } dy\right]\\
=& \displaystyle r_n\int_{ B(t_1,r_1+r_2)} dt_2 \int_{S(t_2,r_2)} dy \mathbb{I}_{\{y \in B(t_1,r_1)\}}\left[ \int_{\R^{2(n-3)}} dt_3 \ldots dt_{n-1} \prod_{i=3}^{n-1} \mathbb{I}_{\{t_i \in B(y,r_i)\}} \right] \\
=& \displaystyle r_n \pi^{n-3}\prod_{i=3}^{n-1}r_i^2\int_{ B(t_1,r_1+r_2)} dt_2 \int_{S(t_2,r_2)} \mathbb{I}_{\{ y\in B(t_1,r_1)\}}dy  =r_n \pi^{n-3}\prod_{i=3}^{n-1}r_i^2 \int_{B(t_1,r_1)}dy \int_{S(y,r_2)}dt_2\\
= & \displaystyle 2\pi^{n-1} r_n r_1^2 r_2  \prod_{i=3}^{n-1}r_i^2.
\end{align*}
 The result follows by summing up three terms in (\ref{genen2}).
\end{proof}
From the above lemma, we can apply Proposition \ref{maint} to deduce the following corollary.
\begin{corollary}
Consider  $X_i(t), 1\leq i \leq n,$ being independent copies of a  two-dimensional Gaussian field $X$  satisfying Assumption (A). Then as $u$ tends to infinity,
\begin{equation}
 \PP\left(\underset{t\in S}{\max}  \underset{1\leq i \leq n}{\min} X_i(t) \geq u \right)= u^{2-n}\varphi^n(u) \left[ \frac{\lambda_2(S)}{2\pi} \left(n+\frac{n(n-1)\pi}{4}\right) +o(1)\right].
\end{equation}
\end{corollary}

\subsection{Proof of Proposition \ref{volu}}\label{detail}
As the readers can see in Subsections 4.2 and 4.3, the basic ideas of the proofs are Fubini theorem  to change the order of the variables in the integrals, and Steiner formula to calculate the area (length) of the $\epsilon$- neighborhood of some set. In general, these ideas are still useful. Let us introduce Weyl tube formula that is a generalization of Steiner formula in higher dimension, and also Crofton formula that will be used later.
\subsubsection{Preliminaries on Weyl tube formula and  Crofton formula \quad\quad\quad\quad\quad\quad\quad\quad\quad\quad}

- \textbf{Weyl tube formula}: Let $M$ be an $m$-dimensional manifold with
positive reach or critical radius (see \cite{MR2319516}) embedded in $R^d$ which is endowed with the canonical Riemannian
structure on $R^d$. Then for any positive $\epsilon$ less than the critical radius of $M$, the Lebesgue volume of the $\epsilon$- neighborhood of $M$ in $R^d$ is given by
\begin{equation}\label{weylt}
\lambda_d\left( M^{+\epsilon}\right)=\sum_{j=0}^m \epsilon^{d-j}\omega_{d-j}\mu_j(M),
\end{equation}

where $\mu_0(M),\mu_1(M),\ldots,\mu_d(M)$ are the intrinsic Killing-Lipschitz curvatures of $M$, that do not depend on the ambient space $\R^d$.
Note that when $M$ is convex, then its critical radius equals to infinity, therefore the above Weyl tube fomula holds true for any positive $\epsilon$. 

- \textbf{Crofton formula}: Borrowing the notations from \cite{MR1608265} and \cite{MR3842168}, let $\mbox{Gr}(d, k)$ be the Grassmanian of all $k-$ dimensional linear subspaces of $\R^d$ with the invariant Haar measure $\nu^d_k$.  Let $\mbox{Graff}(d, k)$ be the affine Grassmannian of all $k$-dimensional affine subspaces of $\R^d$. We define the measure $\lambda^d_k$ on $\mbox{Graff}(d, k)$ that is
invariant under the group of Euclidean motions as follows.

Given a $k$-dimensional affine subspaces $V^* \in \mbox{Graff}(d, k)$. Let $V^{*,\bot}$ be the maximal
linear subspace of $\R^d$ orthogonal to $V^*$ and containing the origin. There
is a unique maximal linear subspace $V \in \mbox{Gr}(d, k)$ orthogonal to $V^{*,\bot}$. It means that $V$ is the $k$-dimensional linear subspace parallel to $V^*$. Denote $p$ by the intersection point between $V^*$ and $V^{*,\bot} =V^{\bot}$. Thus, any $V^* \in \mbox{Graff}( d, k)$ corresponds one-to-one to a pair $(V, p) \in \mbox{Gr}(d, k)\times V^{\bot}$, in the sense that $V^*=V+p$, where the plus symbol stands for the translation.Then, the measure $\lambda^d_k$ on $\mbox{Graff}(d, k)$ is defined as, for any real-valued measurable
function $f$ on $\mbox{Graff}(d, k)$,
$$\int_{\mbox{Graff}(d, k)} f(V^*) d\lambda^d_k(V^*) =  \int_{\mbox{Gr}(d, k)}\int_{V^{\bot}} f(V+p) d\nu^d_k(V) dp,$$
where $dp$ denotes the ordinary Lebesgue measure on $V^{\bot}\cong \R^{d-k}$.

In a special case, let $M$ be a suitable compact subset of $\R^d$ and function $f$ be the Killing-Lipschitz curvatures of the intersection $M\cap V^*$, we have the Crofton formula as
\begin{equation}\label{croft}
\int_{\mbox{Graff}(d, k)} \mu_j(M\cap V^*) d\lambda^d_k(V^*) = \begin{bmatrix}
d-k+j\\j
\end{bmatrix} \mu_{d-k+j}(M),
\end{equation}
where the flag symbol stands for
$$ \begin{bmatrix}
m\\n
\end{bmatrix}=\frac{\omega_m}{\omega_n\omega_{m-n}} \binom{m}{n}.$$
\subsubsection{Detailed proof}
It is clear that
\begin{align*}
I &\displaystyle = \lambda_{(n-1)d} \left((t_2,\ldots,t_n) \in \R^{d (n-1)} :\; \underset{1\leq i \leq n}{\cap} B(t_i,r_i) \neq \emptyset \right)\\
&= \displaystyle \int_{\R^{(n-1)d}} \mathbb{I}_{\{\underset{1\leq i \leq n}{\cap} B(t_i,r_i) \neq \emptyset \}} dt_2\ldots dt_{n} \notag \\
 &= \displaystyle\int_{\R^{(n-2)d}} \mathbb{I}_{\{\underset{1\leq i \leq n-1}{\cap} B(t_i,r_i) \neq \emptyset \}} dt_2\ldots dt_{n-1} \int_{\R^d}\mathbb{I}_{\{B(t_n,r_n) \cap \left( \underset{1\leq i \leq n-1}{\cap} B(t_i,r_i) \right) \neq \emptyset \}} dt_n\\
 &= \displaystyle\int_{\R^{(n-2)d}} \mathbb{I}_{\{\underset{1\leq i \leq n-1}{\cap} B(t_i,r_i) \neq \emptyset \}} \lambda_d\left( \left( \underset{1\leq i \leq n-1}{\cap} B(t_i,r_i) \right) ^{+r_n}\right) dt_2\ldots dt_{n-1} 
\end{align*}

By Weyl tube formula in (\ref{weylt}),
$$\lambda_d\left( \left( \underset{1\leq i \leq n-1}{\cap} B(t_i,r_i) \right) ^{+r_n}\right) =\sum_{k_n=0}^d r_n^{k_n}\omega_{k_n} \mu_{d-k_n}\left(\underset{1\leq i \leq n-1}{\cap} B(t_i,r_i)\right).$$

Substituting this expansion in the integral, we have
\begin{align*}
I &\displaystyle =  \sum_{k_n=0}^d r_n^{k_n}\omega_{k_n} \int_{\R^{(n-2)d}} \mu_{d-k_n}\left(\underset{1\leq i \leq n-1}{\cap} B(t_i,r_i)\right) dt_2\ldots dt_{n-1}\\
&\displaystyle =  \sum_{k_n=0}^d r_n^{k_n}\omega_{k_n} \int_{\R^{(n-2)d}}  dt_2\ldots dt_{n-1} \int_{\mbox{Graff}(d,k_n)}\mu_0 \left(\underset{1\leq i \leq n-1}{\cap} B(t_i,r_i) \cap V^* \right) d\lambda^d_{k_n}(V^*)\\
&\displaystyle =  \sum_{k_n=0}^d r_n^{k_n}\omega_{k_n} \int_{\R^{(n-2)d}}  dt_2\ldots dt_{n-1} \int_{\mbox{Gr}(d,k_n)}d\nu^d_{k_n}(V) \int_{V^{\bot}} \mu_0\left(\underset{1\leq i \leq n-1}{\cap} B(t_i,r_i) \cap (V+p) \right) dp  \\
&\displaystyle =  \sum_{k_n=0}^d r_n^{k_n}\omega_{k_n} \int_{\R^{(n-2)d}}  dt_2\ldots dt_{n-1} \int_{\mbox{Gr}(d,k_n)}d\nu^d_{k_n}(V) \int_{V^{\bot}}  \mathbb{I}_{\{ p\in \underset{1\leq i \leq n-1}{\cap} B(t_i,r_i)\mid_{V^{\bot}} \}} dp , \\
\end{align*}
where the second line follows from Crofton formula (\ref{croft}) applied to the case $j=0$; the third line follows from the definition of the measure $\lambda^d_{k_n}$ and the last line follows from the fact that the intersection $\underset{1\leq i \leq n-1}{\cap} B(t_i,r_i) \cap (V+p)$ is empty or  a non-empty convex, with  Euler characteristic  0 or 1.  The value  depends on whether the point $p$ is on the orthogonal projection of $\underset{1\leq i \leq n-1}{\cap} B(t_i,r_i)$ on the subspace $V^{\bot}$.

By the  Fubini theorem, we continue as
\begin{align*}
I &\displaystyle = \sum_{k_n=0}^d r_n^{k_n}\omega_{k_n} \int_{\R^{(n-3)d}}  dt_2\ldots dt_{n-2} \int_{\mbox{Gr}(d,k_n)}d\nu^d_{k_n}(V) \times \\
  &\displaystyle \quad\quad\quad \int_{V^{\bot}}  \mathbb{I}_{\{ p\in \underset{1\leq i \leq n-2}{\cap} B(t_i,r_i)\mid_{V^{\bot}} \}} dp \int_{\R^d}\mathbb{I}_{\{B(t_{n-1},r_{n-1}) \cap \left( \underset{1\leq i \leq n-2}{\cap} B(t_i,r_i) \cap (V+p)\right) \neq \emptyset \}} dt_{n-1}.\\
\end{align*}
Here  using again Weyl tube formula (\ref{weylt}), we have
\begin{align*}
I &=\displaystyle \sum_{k_n=0}^d r_n^{k_n}\omega_{k_n} \int_{\R^{(n-3)d}}  dt_2\ldots dt_{n-2} \int_{\mbox{Gr}(d,k_n)}d\nu^d_{k_n}(V) \times \\
  &\displaystyle \quad\quad \int_{V^{\bot}}  \mathbb{I}_{\{ p\in \underset{1\leq i \leq n-2}{\cap} B(t_i,r_i)\mid_{V^{\bot}} \}}  \sum_{k_{n-1}=d-k_n}^d r_{n-1}^{k_{n-1}}\omega_{k_{n-1}}\mu_{d-k_{n-1}}\left(\underset{1\leq i \leq n-2}{\cap} B(t_i,r_i) \cap (V+p) \right) dp \\
   &=\displaystyle \sum_{k_n=0}^d r_n^{k_n}\omega_{k_n}  \sum_{k_{n-1}=d-k_n}^d r_{n-1}^{k_{n-1}}\omega_{k_{n-1}} \int_{\R^{(n-3)d}}  dt_2\ldots dt_{n-2} \times\\
  &\displaystyle \quad\quad\int_{\mbox{Gr}(d,k_n)}d\nu^d_{k_n}(V) \int_{V^{\bot}}\mu_{d-k_{n-1}}\left(\underset{1\leq i \leq n-2}{\cap} B(t_i,r_i) \cap (V+p) \right) dp\\
 &=\displaystyle  \sum_{k_n=0}^d r_n^{k_n}\omega_{k_n}  \sum_{k_{n-1}=d-k_n}^d r_{n-1}^{k_{n-1}}\omega_{k_{n-1}} \int_{\R^{(n-3)d}}  dt_2\ldots dt_{n-2} \times\\
&  \displaystyle \quad\quad \int_{\mbox{Graff}(d, k)} \mu_{d-k_{n-1}}\left(\underset{1\leq i \leq n-2}{\cap} B(t_i,r_i) \cap V^* \right) d\lambda^d_k(V^*)\\
    &=\displaystyle  \sum_{k_n=0}^d r_n^{k_n}\omega_{k_n}  \sum_{k_{n-1}=d-k_n}^d r_{n-1}^{k_{n-1}}\omega_{k_{n-1}}\begin{bmatrix}
d-k_n+(d-k_{n-1})\\d-k_{n-1}
\end{bmatrix} \times\\
  &\displaystyle \quad\quad  \int_{\R^{(n-3)d}}   \mu_{d-k_n+(d-k_{n-1})}\left(\underset{1\leq i \leq n-2}{\cap} B(t_i,r_i) \right)dt_2\ldots dt_{n-2},
\end{align*}
where the last line follows from  the Crofton formula (\ref{croft}).

In summary, we have proved that
\begin{align*}
&\displaystyle \int_{\R^{(n-2)d}} \mu_{d-k_n}\left(\underset{1\leq i \leq n-1}{\cap} B(t_i,r_i)\right) dt_2\ldots dt_{n-1}
=\displaystyle  \sum_{k_{n-1}=d-k_n}^d r_{n-1}^{k_{n-1}}\omega_{k_{n-1}}\begin{bmatrix}
d-k_n+d-k_{n-1}\\d-k_{n-1}
\end{bmatrix}   \\
&\displaystyle \quad\quad\quad\quad\quad\quad\quad \quad\quad \quad\quad \times \int_{\R^{(n-3)d}}   \mu_{d-(k_n+k_{n-1}-d)}\left(\underset{1\leq i \leq n-2}{\cap} B(t_i,r_i) \right)dt_2\ldots dt_{n-2},
\end{align*}
then using this argument repeatedly, we obtain that
\begin{align*}
I=&\displaystyle \sum_{k_n=0}^d r_n^{k_n}\omega_{k_n} \sum_{k_{n-1}=d-k_n}^d r_{n-1}^{k_{n-1}}\omega_{k_{n-1}}\begin{bmatrix}
d-k_n+d-k_{n-1}\\d-k_{n-1}
\end{bmatrix} \times\\
&\displaystyle  \sum_{k_{n-2}=d-k_n+d-k_{n-1}}^d r_{n-2}^{k_{n-2}}\omega_{k_{n-2}}\begin{bmatrix}
d-k_n+d-k_{n-1}+d-k_{n-2}\\d-k_{n-2}
\end{bmatrix}  \times \ldots \times \\
&\displaystyle  \sum_{k_{2}=(n-2)d-(k_n+k_{n-1}+\ldots+k_3)}^d r_{2}^{k_{2}}\omega_{k_{2}}\begin{bmatrix}
(d-k_n)+\ldots+(d-k_{2})\\d-k_{2}
\end{bmatrix} \mu_{(n-1)d-\sum_{i=2}^n k_i}\left( B(t_1,r_1)\right).
\end{align*}

By comparing two formulas
$$\lambda_d\left( B(t_1,r_1)^{+\epsilon}\right)=\sum_{j=0}^d \epsilon^{d-j}\omega_{d-j}\mu_j(M)$$
and 
$$\lambda_d\left( B(t_1,r_1)^{+\epsilon}\right)=\lambda_d\left( B(t_1,r_1+\epsilon)\right)=\omega_d(r_1+\epsilon)^d=\omega_d\sum_{j=0}^d \binom{d}{j} \epsilon^{d-j}r_1^j,$$
it is clear that
$$\mu_j\left( B(t_1,r_1)\right)= \frac{\omega_d}{\omega_{d-j}}\binom{d}{j} r_1^j.$$ 

Thus we have
\begin{align*}
I=&\displaystyle \sum_{k_n=0}^d  \sum_{k_{n-1}=d-k_n}^d \ldots \sum_{k_{2}=(n-2)d-(k_n+k_{n-1}+\ldots+k_3)}^d r_1^{(n-1)d-\sum_{i=2}^n k_i} \times\\
&\displaystyle \prod_{i=2}^n \left( r_i^{k_i}\omega_{k_i} \begin{bmatrix}
(d-k_n)+\ldots+(d-k_{i})\\d-k_{i}
\end{bmatrix}\right) \frac{\omega_d}{\omega_{\sum_{i=2}^n k_i -(n-2)d}}\binom{d}{(n-1)d-\sum_{i=2}^n k_i}.
\end{align*}
Observe that
\begin{align*}
&\displaystyle \frac{\omega_d}{\omega_{\sum_{i=2}^n k_i -(n-2)d}}\binom{d}{(n-1)d-\sum_{i=2}^n k_i} \prod_{i=2}^n  \begin{bmatrix}
(d-k_n)+\ldots+(d-k_{i})\\d-k_{i}
\end{bmatrix}\\
 = &\displaystyle \frac{\omega_d}{\omega_{\sum_{i=2}^n k_i -(n-2)d}}\binom{d}{(n-1)d-\sum_{i=2}^n k_i} \prod_{i=2}^{n-1} \frac{\omega_{(d-k_n)+\ldots+(d-k_{i})} }{\omega_{(d-k_n)+\ldots+(d-k_{i+1})} \times \omega_{d-k_{i}} } \binom{(d-k_n)+\ldots+(d-k_{i})}{d-k_{i}}\\
 = &\displaystyle \frac{\omega_d \omega_{(n-1)d-\sum_{i=2}^n k_i}}{\omega_{\sum_{i=2}^n k_i -(n-2)d} \prod_{i=2}^{n} \omega_{d-k_i}} \times \frac{d!}{ \left[ \sum_{i=2}^n k_i -(n-2)d\right]! .\prod_{i=2}^{n} (d-k_i)!},
\end{align*}
this completes the proof.

\section{Comparing with Euler characteristic method}
In this section, we would like to compare our result given in the main theorem with the prediction given by Euler characteristic method. This prediction is defined in (\ref{ECH}) as
$$(1,0,\ldots,0)R^2\mu(S).$$
Note that  the indexes of the rows and columns of matrix $R$ varies from $0$ to $d$,  the term corresponding to $\mu_d(S)$ (or $\lambda_d(S)$) is
$$\lambda_d(S)b_d\sum_{0\leq h_1\leq  \ldots \leq h_{n-1} \leq d} \frac{\rho_{h_1}}{b_{h_1}} \frac{\rho_{h_2-h_1}}{b_{h_2-h_1}} \ldots \frac{\rho_{d-h_{n-1}}}{b_{d-h_{n-1}}},$$
and it is equivalent to
$$ \sum_{0\leq h_1\leq  \ldots \leq h_{n-1} \leq d} \frac{(2\pi)^{-d/2} \varphi^n(u) u^{d-n} \lambda_d(S) \Gamma(1/2)^{n-1}\Gamma((d+1)/2)}{\Gamma((h_1+1)/2) \Gamma((h_2-h_1+1)/2) \ldots \Gamma((d-h_{n-1}+1)/2)}.$$

To prove that the above sum coincides with the asymptotic formula (\ref{mafomu}), we need to show that
\begin{align*}
&\displaystyle \sum_{0\leq h_1\leq  \ldots \leq h_{n-1} \leq d} \frac{ \Gamma(1/2)^{n-1}\Gamma((d+1)/2)}{\Gamma((h_1+1)/2) \Gamma((h_2-h_1+1)/2) \ldots \Gamma((d-h_{n-1}+1)/2)}\\
=& \displaystyle \sum_{k_n=0}^d  \sum_{k_{n-1}=d-k_n}^d \ldots \sum_{k_{2}=(n-2)d- \sum_{i=3}^n k_i}^d \frac{\omega_d}{\omega_{\sum_{i=2}^n k_i -(n-2)d} \prod_{i=2}^{n} \omega_{d-k_i}} \times \frac{d!}{ \left[ \sum_{i=2}^n k_i -(n-2)d\right]! \prod_{i=2}^{n} (d-k_i)!}.
\end{align*}

Indeed, we rewrite the indices $(h_1,h_2,\ldots,h_{n-1})$ as
$$h_1=d-k_n,\; h_2=(d-k_n)+(d-k_{n-1}), \ldots , h_{n-1}=\sum_{i=2}^n (d-k_i),$$
and  it can be  checked one-by-one that
\begin{align*}
&\displaystyle \frac{ \Gamma(1/2)^{n-1}\Gamma((d+1)/2)}{\Gamma((d-k_n+1)/2) \Gamma((d-k_{n-1}+1)/2) \ldots \Gamma((d-k_2+1)/2)\Gamma \left(\left(d+1-\sum_{i=2}^n \left(d-k_i\right)\right)/2\right)}\\
=& \displaystyle \frac{\omega_d}{\omega_{d-\sum_{i=2}^n (d-k_i)} \prod_{i=2}^{n} \omega_{d-k_i}} \times \frac{d!}{ \left[ d-\sum_{i=2}^n (d-k_i)\right]! \prod_{i=2}^{n} (d-k_i)!}.
\end{align*}
The equality follows easily from Legendre duplication formula (\ref{legendre}) as in Subsection 4.1.

In conclusion, we give a one-term expansion for the conjunction probability of smooth Gaussian fields. It is interesting to see that this expansion coincide in the first term with the heuristic prediction given by Euler characteristic method although they look different at first sight. Since the heuristic prediction consists of $d+1$ terms, it is natural to ask that

\textit{"Could we prove the full validity of Euler characteristic method for the conjunction probability as for the tail distribution of a smooth Gaussian field?"}

We would like to leave this interesting question for future research.
 \section*{Acknowledgements} 
  This work is supported by Grant IM-VAST01-2020.01.  Part of this work was done during the author's post-doctoral fellowship of the Vietnam Institute for Advanced Study in Mathematics in 2016.

\end{document}